\documentclass[12pt]{amsart}
\usepackage{amssymb,latexsym,amsmath,amscd,amsthm,graphicx, color}
\usepackage[all]{xy}
\usepackage{pgf,tikz}
\usepackage{mathrsfs}

\usetikzlibrary{arrows}
\definecolor{uuuuuu}{rgb}{0.26666666666666666,0.26666666666666666,0.26666666666666666}
\definecolor{xdxdff}{rgb}{0.49019607843137253,0.49019607843137253,1.}
\definecolor{ffqqqq}{rgb}{1.,0.,0.}
\definecolor{ffqqqq}{rgb}{1.,0.,0.}
\definecolor{ffxfqq}{rgb}{1.,0.4980392156862745,0.}
\definecolor{xdxdff}{rgb}{0.49019607843137253,0.49019607843137253,1.}
\definecolor{ffqqqq}{rgb}{1.,0.,0.}
\definecolor{ududff}{rgb}{0.30196078431372547,0.30196078431372547,1.}
\definecolor{uuuuuu}{rgb}{0.26666666666666666,0.26666666666666666,0.26666666666666666}
\definecolor{qqwuqq}{rgb}{0.,0.39215686274509803,0.}
\definecolor{zzttqq}{rgb}{0.6,0.2,0.}
\definecolor{xdxdff}{rgb}{0.49019607843137253,0.49019607843137253,1.}
\definecolor{qqqqff}{rgb}{0.,0.,1.}
\definecolor{cqcqcq}{rgb}{0.7529411764705882,0.7529411764705882,0.7529411764705882}
\definecolor{sqsqsq}{rgb}{0.12549019607843137,0.12549019607843137,0.12549019607843137}

\theoremstyle{plain}

\newtheorem{theorem}[subsection]{Theorem}

\newtheorem{prop}[subsection]{Proposition}

\theoremstyle{definition}

\newtheorem{remark}[subsection]{Remark}

\newcommand{\uu}{\cup}
\newcommand{\UU}{\bigcup}

\newcommand{\sci}{\subset}

\newcommand{\set}[1]{\{#1\}}

\newcommand{\ga}{\alpha}

\newcommand{\tit}{\textit}

\newcommand{\D}[1]{\mathbb{#1}}

\newcommand{\te}{\text}

\newcommand{\pa}{\partial}

\newcommand{\tl}{\tilde}
\begin{document}

To appear,  Statistics \& Probability Letters
\title{Quantization coefficients for uniform distributions on the boundaries of regular polygons}


\address{School of Mathematical and Statistical Sciences\\
University of Texas Rio Grande Valley\\
1201 West University Drive\\
Edinburg, TX 78539-2999, USA.}

\email{\{$^1$joel.hansen01, $^2$itzamar.marquez01, $^3$mrinal.roychowdhury, $^4$eduardo.torres05\}\linebreak@utrgv.edu}

\author{$^1$Joel Hansen}
\author{$^2$Itzamar Marquez}
\author{$^3$Mrinal K. Roychowdhury}
\author{$^4$Eduardo Torres}

\subjclass[2010]{60Exx, 94A34.}
\keywords{Uniform distribution, optimal sets, quantization error, quantization coefficient, regular polygon}

\date{}
\maketitle

\pagestyle{myheadings}\markboth{J. Hansen, I. Marquez, M.K. Roychowdhury, and E. Torres}
{Quantization coefficients for uniform distributions on the boundaries of regular polygons}

\begin{abstract}
In this paper, we give a general formula to determine the quantization coefficients for uniform distributions defined on the boundaries of different regular $m$-sided polygons inscribed in a circle. The result shows that the quantization coefficient for the uniform distribution on the boundary of a regular $m$-sided polygon inscribed in a circle is an increasing function of $m$, and approaches to the quantization coefficient for the uniform distribution on the circle as $m$ tends to infinity.
\end{abstract}

\section{Introduction}

Let $P$ be a Borel probability measure on the $d$-dimensional Euclidean space $\D R^d$ equipped with the Euclidean norm $\|\cdot\|$, where $d\geq 1$. For a finite set $\ga\sci \D R^d$, the \tit{cost} or \tit{distortion error} for $P$ with respect to the set $\ga$, denoted by $V(P; \ga)$, is defined by
\[V(P; \ga):= \int \min_{a\in\alpha} \|x-a\|^2 dP(x).\]
Then, for $n\in \D N$, the $n$th quantization error for $P$, denoted by $V_n:=V_n(P)$, is defined by
\begin{equation*} \label{eq0} V_n:=V_n(P)=\inf \Big\{V(P; \ga) : \ga\sci \D R^d, 1\leq \text{ card}(\ga) \leq n \Big\}.\end{equation*}
A set $\ga$ for which the infimum is achieved and contains no more than $n$ points is called an \textit{optimal set of $n$-means}. It is well-known that for a continuous probability measure an optimal set of $n$-means always contains exactly $n$ elements. If $P$ is the probability distribution, then an optimal set of $n$-means is denoted by $\ga_n:=\ga_n(P)$. Optimal sets of $n$-means for different probability distributions were determined by several authors, for example, see \cite{CR1, CR2, DR1, DR2, GL2, L, R1, R2, R3, R4, R5, R6, RR1, RS}. It has broad applications in engineering and technology (see \cite{GG, GN, Z}). For any $s\in (0, +\infty)$, the number \[\lim_{n\to \infty} n^{\frac 2 s} V_n(P),\]
if it exists, is called the $s$-dimensional \tit{quantization coefficient} for $P$. Bucklew and Wise (see \cite{BW}) showed that for a Borel probability measure $P$ with non-vanishing absolutely continuous part the quantization coefficient exists as a finite positive number.
For some more details interested readers can also see \cite{GL1, P}.
Let $E(X)$ represent the expected value of a random variable $X$ associated with a probability distribution $P$. Let $\ga$ be an optimal set of $n$-means for $P$, and $a\in \ga$. Then, it is well-known that
$a=E(X : X \in M(a|\ga)),$
where $M(a|\ga)$ is the Voronoi region of $a\in \ga , $ i.e.,  $M(a|\ga)$ is the set of all elements $x$ in $\D R^d$ which are closest to $a$ among all the elements in $\ga$ (see \cite{GG, GL1}).

From the work of Rosenblatt and Roychowdhury (see \cite{RR2}), it is known that the quantization coefficient for the uniform distribution on a unit circle is $\frac {\pi^2}3$; on the other hand, from the work of Pena et al. (see \cite{PRRSS}), it is known that the quantization coefficient for the uniform distribution on the boundary of a regular hexagon inscribed in a unit circle is $3$. Notice that a regular $m$-sided polygon inscribed in a circle tends to the circle as $m$ tends to infinity. Pena et al. conjectured that the quantization coefficient for the uniform distribution on the boundary of a regular $m$-sided polygon inscribed in a circle is an increasing function of $m$ (see \cite{PRRSS}), and approaches to the quantization coefficient for the uniform distribution on the circle as $m$ tends to infinity. In this paper, we prove that the conjecture is true.

The arrangement of the paper is as follows: First, we prove a theorem Theorem~\ref{theo00}, which gives a technique how to calculate the optimal sets of $n$-means and the $n$th quantization errors for all positive integers $n$ for a uniform distribution defined on any line segment. Next, let $P$ be the uniform distribution defined on the boundary of a regular $m$-sided polygon inscribed in a unit circle. In Proposition~\ref{prop22}, for $k\geq 2$, we determine the optimal set of $mk$-means and the $mk$th quantization error for the probability distribution $P$. Then, with the help of the proposition, in Theorem~\ref{theo01}, we have shown that the quantization coefficient for $P$ exists, and equals $\frac{1}{3} m^2 \sin ^2\left(\frac{\pi }{m}\right)$, i.e.,
\[\lim_{n\to \infty} n^2 V_n(P)=\frac{1}{3} m^2 \sin ^2\left(\frac{\pi }{m}\right).\]
Notice that $\frac{1}{3} m^2 \sin ^2\left(\frac{\pi }{m}\right)$ is an increasing function of $m$, and $\lim\limits_{m\to \infty} \frac{1}{3} m^2 \sin ^2\left(\frac{\pi }{m}\right)=\frac{\pi^2}{3},$
which is the quantization coefficient for a uniform distribution on the unit circle (see \cite{RR2}). Thus, the result in this paper, shows that the conjecture given by Pena et al. in \cite{PRRSS} is true.

\section{Main result}

In this section, first we give some basic definitions.


By the position vector of a point $A$ with coordinates $(a_1, a_1)$ it is meant that $\overrightarrow{OA}=(a_1, a_2):=a_1i+a_2j$, where $i$ and $j$ are two unit vectors in the positive directions of the $x_1$- and $x_2$-axes, respectively. The squared Euclidean distance between the two points $a:=(a_1, a_2)$ and $b:=(b_1, b_2)$ is denoted by $\rho(a, b)$, i.e.,  $\rho(a, b):=\|(a_1, a_2)-(b_1, b_2)\|^2=(a_1-b_1)^2 +(a_2-b_2)^2$. Let $D$ be the boundary of the Voronoi regions of two points $A$ and $B$ belonged to an optimal set of $n$-means for some positive integer $n$. Then, as we know that the boundary of the Voronoi regions of any two points is the perpendicular bisector of the line segment joining the points, we have
$|\overrightarrow{DA}|=|\overrightarrow{DB}|, \te{ i.e., } (\overrightarrow{DA})^2=(\overrightarrow{DB})^2$ implying
$(  a-  d)^2=(  b-  d)^2$, i.e., $\rho(  d,   a)-\rho(  d,   b)=0$, where $a, b, d$ are, respectively, the position vectors of the points $A, B, D$. We call such an equation a \tit{canonical equation}.

Let us now give the following theorem.

\begin{theorem} \label{theo00}
Let $AB$ be a line segment joining the two points $A$ and $B$ given by the position vectors $a:=(a_1, a_2)$ and $b:=(b_1, b_2)$, respectively. Let $\mu$ be a uniform distribution on $AB$. Let $M(t)$ be the parametric representation of $AB$ for $0\leq t\leq 1$ such that $M(0)=a$, and $M(1)=b$. Let $D_1$ and $D_2$ be two points on the segment $AB$ at distances $r_1$ and $r_2$ from $A$ and $B$, respectively (see Figure~\ref{Fig1}). Then, for each $n\in \D N$, the optimal set of $n$-means for $\mu$ on the segment $D_1D_2$ is given by
\[\ga_n(\mu,  \, D_1D_2):=\Big\{M(\frac {r_1}{\ell}+\frac {2j-1}{2n} (1-\frac{r_2}{\ell}-\frac{r_1}{\ell})) : 1\leq j\leq n \Big\},\]
with the  $n$th quantization error for $\mu$ on the segment $D_1D_2$,
\[V_n(\mu, \, D_1D_2):=n \int_{\frac {r_1}{\ell}}^{\frac {r_1}{\ell}+\frac {1}{n} (1-\frac{r_2}{\ell}-\frac{r_1}{\ell})}\rho\Big(M(t), M(\frac {r_1}{\ell}+\frac {1}{2n} (1-\frac{r_2}{\ell}-\frac{r_1}{\ell}))\Big) d\mu,\]
where $\ell$ is the length of the line segment $AB$.
\end{theorem}

\begin{figure}
\begin{tikzpicture}[line cap=round,line join=round,>=triangle 45,x=1.0cm,y=1.0cm]
\clip(-4.434877892537599,-0.5065236901434486) rectangle (11.67941711302465,5.21943965200402);
\draw [line width=0.5pt,color=ffqqqq] (-4.,1.)-- (9.,5.);
\draw (-4.331619416758946,1.0466997550116428) node[anchor=north west] {$A(M(0)=a)$};
\draw (8.775247841033814,5.110995388362796) node[anchor=north west] {$B(M(1)=b)$};
\draw (6.616698485521949,4.383174795378465) node[anchor=north west] {$D_2(d_2=M(1-\frac{r_2}{\ell}))$};
\draw (-2.425873138919643,1.624555378619519) node[anchor=north west] {$D_1(d_1=M(\frac{r_1}{\ell}))$};
\draw [line width=0.5 pt,dotted] (0.,0.)-- (9.,5.);
\draw [line width=0.5 pt,dotted] (0.,0.)-- (-4.,1.);
\draw (-0.07285987750580951,0.02995135022337276) node[anchor=north west] {$O$};
\draw (1.7244421366922663,3.055424263641221) node[anchor=north west] {$|AB|=\ell$};
\draw (-4.331619416758946,2.124555378619519) node[anchor=north west] {$|AD_1|=r_1$};
\draw (6.616698485521949,5.383174795378465) node[anchor=north west] {$|BD_2|=r_2$};
\draw [->,line width=0.5pt,dotted] (0.,0.) -- (6.,0.);
\draw [->,line width=0.5pt,dotted] (0.,0.) -- (-3.,0.);
\draw [->,line width=0.5pt,dotted] (0.,0.) -- (0.,4.);
\draw [->,line width=0.5 pt,dotted] (0.,0.) -- (0.02,-.5377855623607878);
\begin{scriptsize}
\draw [fill=xdxdff] (-4.,1.) circle (1.5pt);
\draw [fill=ududff] (9.,5.) circle (1.5pt);
\draw [fill=xdxdff] (-2.0015135135135136,1.614918918918919) circle (1.5pt);
\draw [fill=xdxdff] (7.016972972972972,4.389837837837838) circle (1.5pt);
\draw [fill=xdxdff] (0.,0.) circle (1.5pt);
\end{scriptsize}
\end{tikzpicture}
\caption{ }
\label{Fig1}
\end{figure}

\begin{proof}
Since $\ell$ is the length of the line segment $AB$, the probability density function (pdf) $f$ of the uniform distribution $\mu$ on $AB$ is given by $f(x_1, x_2)=\frac 1{\ell}$ for all $(x_1, x_2)\in AB$, and zero otherwise.
Let $s$ represent the distance of any point on $AB$ from the point $A$. Then, we have $d\mu=d\mu(s)=\mu(ds)=f(x_1, x_2) ds=\frac 1 {\ell} \,ds$. Notice that $ds=\sqrt{(\frac {dx_1}{dt})^2 +(\frac {dx_2}{dt})^2}\,dt=\ell \,dt$ yielding $d\mu=dt$. By the hypothesis, the parametric representation of the line segment $AB$ is $M(t)$ for $0\leq t\leq 1$ with $M(0)=a$ and $M(1)=b$. Hence, the parameters for the points $D_1$ and $D_2$, which are at distances $r_1$ and $r_2$ from $A$ and $B$ are, respectively, given by $t=\frac {r_1}{\ell}$ and $t=1-\frac {r_2}{\ell}$, i.e., if $d_1$ and $d_2$ are the position vectors of the points $D_1$ and $D_2$ (see Figure~\ref{Fig1}), then we have
\[d_1=M(\frac {r_1}{\ell}), \te{ and } d_2=M(1-\frac {r_2}{\ell}).\]
In fact, we can identify the line segment $D_1D_2$ by its parameters in the closed interval $[\frac {r_1}{\ell}, 1-\frac {r_2}{\ell}].$ By \cite{RR2}, we know that the optimal set of $n$-means with respect to an uniform distribution in the closed interval  $[\frac {r_1}{\ell}, 1-\frac {r_2}{\ell}]$ is given by the set
\[\Big \{\frac {r_1}{\ell}+\frac {2j-1}{2n}(1-\frac {r_2}{\ell}-\frac {r_1}{\ell}) : 1\leq j\leq n\Big\}.\]
Hence, the optimal set of $n$-means for $\mu$ on the segment $D_1D_2$, is given by
\[\ga_n(\mu,  \, D_1D_2):=\Big\{M(\frac {r_1}{\ell}+\frac {2j-1}{2n} (1-\frac{r_2}{\ell}-\frac{r_1}{\ell})) : 1\leq j\leq n \Big\}.\]
If $V_n(\mu, \, D_1D_2)$ is the corresponding quantization error, then we have
\begin{align*} V_n(\mu, \, D_1D_2)=n\Big(\te{Quantization error due to the point } M(\frac {r_1}{\ell}+\frac {1}{2n} (1-\frac{r_2}{\ell}-\frac{r_1}{\ell}))\Big).
\end{align*}
Again, notice that any point on the line segment $D_1D_2$ is given by $M(t)$ for $\frac{r_1}{\ell}\leq t\leq 1-\frac{r_2}{\ell}$, and the parameters for the points at which the boundary of the Voronoi region of $M(\frac {r_1}{\ell}+\frac {1}{2n} (1-\frac{r_2}{\ell}-\frac{r_1}{\ell}))$ cuts the segment $D_1D_2$ are given by $t=\frac{r_1}{\ell}$, and $t=\frac {r_1}{\ell}+\frac {1}{n} (1-\frac{r_2}{\ell}-\frac{r_1}{\ell})$. Hence, we have
\[V_n(\mu, \, D_1D_2)=n \int_{\frac {r_1}{\ell}}^{\frac {r_1}{\ell}+\frac {1}{n} (1-\frac{r_2}{\ell}-\frac{r_1}{\ell})}\rho\Big(M(t), M(\frac {r_1}{\ell}+\frac {1}{2n} (1-\frac{r_2}{\ell}-\frac{r_1}{\ell}))\Big) d\mu.\]
Thus, the proof of the theorem is complete.
\end{proof}
Let the equation of the unit circle be $x_1^2+x_2^2=1$. Let $A_1A_2A_3\cdots A_m$ be a regular $m$-sided polygon for some $m\geq 3$ inscribed in the circle. Without any loss of generality due to rotational symmetry, we can always assume that the vertex $A_1$ lies on the $x_1$-axis, i.e., the vertex $A_1$ is the point where the circle intersects the positive direction of the $x_1$-axis. Again, notice that each side of the regular $m$-sided polygon subtends a central angle of radian $\frac {2\pi}m$. Thus, the position vectors $\tilde a_j$ of the vertices $A_j$ are given by $\tilde a_j=(\cos \frac{2 \pi }{m} (j-1), \sin\frac{2 \pi }{m} (j-1))$ for $1\leq j\leq m$. Let $\ell$ be the length of each side of the polygon, then we have
\begin{equation} \label{eq100} \ell=\|\tilde a_m-\tilde a_{m-1}\|=\|\tilde a_{m-1}-\tilde a_{m-2}\|=\cdots=\|\tilde a_2-\tilde a_1\|=2 \sin \frac{\pi }{m}.
\end{equation}
Let $L$ be the boundary of the polygon. Then, we can write
\[L=\UU_{j=1}^m L_j,\]
where $L_j$ is the side $A_jA_{j+1}$, and $A_{m+1}$ is identified as the vertex $A_1$. Then, for $1\leq j\leq m$, we can write
\[L_j:=A_jA_{j+1}=\set{M_j :  0\leq t\leq 1}, \te{ where } M_j=\tl a_{j+1}t+(1-t)\tl a_j.\]
Notice that $M_j$ is a function of $t$, and any point on the side $A_jA_{j+1}$ can be represented by $M_j:=M_j(t)$ for $0\leq t\leq 1$. Thus, we see that $M_j(0)=\tilde a_j$, and $M_j(1)=\tilde a_{j+1}$ for $1\leq j\leq m$.
Let $P$ be the uniform distribution defined on the boundary $L$ of the polygon. Then, the probability density function (pdf) $f$ of the uniform distribution $P$ is given by $f(x_1, x_2)=\frac 1{m\ell}$ for all $(x_1, x_2)\in L$, and zero otherwise.
Let $s$ represent the distance of any point on $L$ from the vertex $A_1$ tracing along the boundary $L$ in the counterclockwise direction. Then, we have $dP=dP(s)=P(ds)=f(x_1, x_2) ds=\frac 1 {m\ell} ds$. For $1\leq j\leq m$, on each $L_j$,  we have $ds=\sqrt{(\frac {dx_1}{dt})^2 +(\frac {dx_2}{dt})^2}|dt|=\ell |dt|$ yielding $dP(s)=\frac \ell{m\ell} |dt|=\frac 1m |dt|$.

\begin{remark} \label{rem1}
Since $P$ is uniform, and a regular $m$-sided polygon has symmetry of order $m$, it is not difficult to show that an optimal set $\ga_m$ will contain $m$ points, each from the interior of the $m$ angles of the regular $m$-sided polygon; and for any positive integer $k\geq 2$, $\ga_{mk}$ will contain $m$ points, each from the interior of the $m$ angles, and $(k-1)$ points from each side of the regular $m$-sided polygon. Moreover, the following is true: Let $A$ be one of the vertices of the regular $m$-sided polygon, and for $k\geq 2$, let $a$ be an element of an optimal set of $mk$-means that lies in the interior of $\angle A$. Further, let $AA_1$ and $AA_2$ be the two adjacent sides of the vertex $A$. Then, the boundary of the Voronoi region of $a$ will cut $AA_1$ and $AA_2$ at two points $D_1$ and $D_2$ such that $|AD_1|=|AD_2|=r$ for some real $r$ such that $0<r\leq \frac {\ell} 2$, where $\ell$ is the length of the sides of the polygon.
\end{remark}
\begin{prop} \label{prop22}
Let $\ga_n$ be an optimal set of $n$-means such that $n=mk$, where $k\in \D N$, and $k\geq 2$. Let $a_j$ be the points that $\ga_n$ contains from the interior of the angles $A_j$ of the regular $m$-sided polygon, $1\leq j\leq m$. Then,
\[\ga_n=\set{a_j : 1\leq j\leq m}\uu \UU_{j=1}^m \ga_{j, k-1},\]
where\begin{align*}
a_1&=(1-\frac{1}{2} r \sin (\frac{\pi }{m}),0),\\
a_j&=(\frac{1}{4} \cos \frac{2 \pi  (j-1)}{m}(r (\cos (\frac{2 \pi }{m})-1) \csc (\frac{\pi }{m})+4), \sin \frac{2 \pi  (j-1)}{m} (\frac{1}{4} r (\cos (\frac{2 \pi }{m})-1) \csc (\frac{\pi }{m})+1))
\end{align*}
for $2\leq j\leq m$, and
$\ga_{j, k-1}:=\set{M_j(\frac r \ell+\frac {2i-1}{2(k-1)}(1-\frac {2r}{\ell})) : 1\leq i\leq k-1}$ for $1\leq j\leq m$, and
\[r=\frac{4 \sin (\frac{\pi }{m})}{2 (k-1) \sqrt{3 \cos ^2(\frac{\pi }{m})+1}+4}.\] Moreover, the quantization error for $n$-means is given by
\begin{align*}
V_n=\frac{2 \sin ^2(\frac{\pi }{m}) (3 \cos (\frac{2 \pi }{m})+5)}{3 \Big(k \sqrt{6 \cos (\frac{2 \pi }{m})+10}-\sqrt{6 \cos (\frac{2 \pi }{m})+10}+4\Big)^2}.
\end{align*}
\end{prop}

\begin{proof}
Let $\ga_n$ be an optimal set of $n$-means, where $n=mk$ for some positive integer $k\geq 2$. Since $a_j$ are the points that $\ga_n$ contains from the interior of the angles $A_j$, by Remark~\ref{rem1}, due to uniform distribution and symmetry, we can say that there exists a real number $r$, where $0<r\leq \frac {\ell} 2$, such that the boundary of the Voronoi region of each $a_j$ will cut the two adjacent sides at distances $r$ from the vertex $A_j$. Notice that the two adjacent sides of the vertex $A_1$ are $A_mA_1$ and $A_1A_2$ in the polygon. Again, by the hypothesis $a_1$ is the point that $\ga_n$ contains from $\angle A_1$. If the boundary of the Voronoi region of $a_1$ cuts $A_mA_1$ and $A_1A_2$ at $D_1$ and $D_2$, respectively, we have
\[a_1=E (X : X\in D_1A_1\uu A_1D_2)=\frac {\int_{D_1A_1}(x_1, x_2) dP+\int_{A_1D_2}(x_1, x_2) dP}{\int_{D_1A_1}1 dP+\int_{A_1D_2}1 dP},\]
which implies
\[a_1=\frac{\int_{1-\frac{r}{L}}^1 M_n(t) \, dt+\int_0^{\frac{r}{L}} M_1(t) \, dt}{\int_0^{\frac{r}{L}} 1 \, dt+\int_{1-\frac{r}{L}}^1 1 \, dt}=(1-\frac{1}{2} r \sin \left(\frac{\pi }{m}\right),0).\]
Similarly, for $2\leq j\leq m$, we obtain
\[a_j=\frac{\int_{1-\frac{r}{L}}^1 M_{j-1}(t) \, dt+\int_0^{\frac{r}{L}} M_j(t) \, dt}{\int_0^{\frac{r}{L}} 1 \, dt+\int_{1-\frac{r}{L}}^1 1 \, dt}\]
yielding
\[a_j=(\frac{1}{4} \cos \frac{2 \pi  (j-1)}{m}(r (\cos (\frac{2 \pi }{m})-1) \csc (\frac{\pi }{m})+4), \sin \frac{2 \pi  (j-1)}{m} (\frac{1}{4} r (\cos (\frac{2 \pi }{m})-1) \csc (\frac{\pi }{m})+1)).\]
Again, by Remark~\ref{rem1}, $\ga_n$ contains $(k-1)$ points from each side of the regular $m$-sided polygon. For $1\leq j\leq m$, let $\ga_{j, k-1}$ be the optimal set of $(k-1)$-means that $\ga_n$ contains from the side $A_jA_{j+1}$. Recall that the parametric representation of the side $A_jA_{j+1}$ is $M_j(t)$, and the $(k-1)$ means from each side occur due to an uniform distribution on the segment bounded by the two points represented by the parameters $t=\frac{r}{\ell}$ and $t=1-\frac{r}{\ell}$. Hence, by Theorem~\ref{theo00}, we have
\[\ga_{j, k-1}:=\Big\{M_j(\frac r \ell+\frac {2i-1}{2(k-1)}(1-\frac {2r}{\ell})) : 1\leq i\leq k-1\Big\}.\]
To calculate the quantization error, we proceed as follows:
By symmetry, the quantization error contributed by all the points $a_j$ for $1\leq j\leq m$ is given by
\begin{align*} m \int_{D_1A_1\uu A_1D_2}\rho((x_1, x_2), a_1) dP&=2 m\int_{A_1D_2}\rho((x_1, x_2), a_1) dP=2 \int_{0}^{\frac {r}{\ell}}\rho(M_1(t), a_1) dt,
\end{align*}
implying
\begin{equation} \label{eq1}  m \int_{D_1A_1\uu A_1D_2}\rho((x_1, x_2), a_1) dP=\frac{1}{24} r^3 (3 \cos (\frac{2 \pi }{m})+5) \csc (\frac{\pi }{m}).
\end{equation}
Again, by Theorem~\ref{theo00}, the quantization error contributed by all the sets $\ga_{j, k-1}$ for $1\leq j\leq m$ is given by
\[m V_n(\mu, \, \ga_{j, k-1}):=(k-1)\int_{\frac {r}{\ell}}^{\frac {r}{\ell}+\frac {1}{k-1} (1-\frac{2r}{\ell})}\rho\Big(M(t), M(\frac {r}{\ell}+\frac {1}{2(k-1)} (1-\frac{2r}{\ell}))\Big) dt.\]
implying
\begin{equation} \label{eq2} m V_n(\mu, \, \ga_{j, k-1})=\frac 1{3(k-1)^2}\csc (\frac{\pi }{m}) (\sin (\frac{\pi }{m})-r)^3.
\end{equation}
Hence, by \eqref{eq1} and \eqref{eq2}, the quantization error for $n$-means is given by
\begin{align*}
V_n= \frac{1}{24} \csc (\frac{\pi }{m})\Big(r^3 (3 \cos(\frac{2 \pi }{m})+5)+\frac 8{(k-1)^2}(\sin(\frac{\pi }{m})-r)^3\Big).
\end{align*}
Notice that for a given $k$, the quantization error $V_n$ is a function of $r$. Solving $\frac{\pa V_n}{\pa r}=0$, we have $r=\frac{4 \sin (\frac{\pi }{m})}{2 (k-1) \sqrt{3 \cos ^2(\frac{\pi }{m})+1}+4}$. Putting $r=\frac{4 \sin (\frac{\pi }{m})}{2 (k-1) \sqrt{3 \cos ^2(\frac{\pi }{m})+1}+4}$, we have
\[V_n=\frac{2 \sin ^2(\frac{\pi }{m}) (3 \cos (\frac{2 \pi }{m})+5)}{3 \Big(k \sqrt{6 \cos (\frac{2 \pi }{m})+10}-\sqrt{6 \cos (\frac{2 \pi }{m})+10}+4\Big)^2}.\]
Thus, the proof the proposition is complete.
\end{proof}

Let us now prove the following theorem.

\begin{theorem}\label{theo01}
Let $P$ be the uniform distribution on the boundary of a regular $m$-sided polygon inscribed in a unit circle. Then, the quantization coefficient for $P$ exists as a finite positive number which equals $\frac{1}{3} m^2 \sin ^2(\frac{\pi }{m})$, i.e.,
$\lim\limits_{n\to \infty} n^2 V_n=\frac{1}{3} m^2 \sin ^2(\frac{\pi }{m}).$
\end{theorem}

\begin{proof} Let $n\in \D N$ be such that $n\geq 2m$. Then, there exists a unique positive integer $\ell(n)\geq 2$ such that $m\ell(n)\leq n<m(\ell(n)+1)$. Then,
\begin{equation} \label{eq46} (m\ell(n))^2V_{m(\ell(n)+1)}<n^2 V_n<(m(\ell(n)+1))^2V_{m\ell(n)}.
\end{equation}
We have
\begin{align*}
&\lim_{n \to \infty} (m\ell(n))^2V_{m(\ell(n)+1)}\\
&=\underset{\ell(n)\to \infty }{\text{lim}}(m \ell(n))^2 \frac{2 \sin ^2(\frac{\pi }{m}) (3 \cos (\frac{2 \pi }{m})+5)}{3 \Big((\ell(n)+1) \sqrt{6 \cos (\frac{2 \pi }{m})+10}-\sqrt{6 \cos (\frac{2 \pi }{m})+10}+4\Big)^2}=\frac{1}{3} m^2 \sin ^2(\frac{\pi }{m}),\end{align*}
and
\begin{align*}
&\lim_{n \to \infty} (m(\ell(n)+1))^2V_{m\ell(n)}\\
&=\underset{\ell(n)\to \infty }{\text{lim}}(m(\ell(n)+1))^2\frac{2 \sin ^2(\frac{\pi }{m}) (3 \cos (\frac{2 \pi }{m})+5)}{3 \Big(\ell(n) \sqrt{6 \cos (\frac{2 \pi }{m})+10}-\sqrt{6 \cos (\frac{2 \pi }{m})+10}+4\Big)^2}=\frac{1}{3} m^2 \sin ^2(\frac{\pi }{m}),
\end{align*}
and hence, by \eqref{eq46} we have
$\mathop{\lim}\limits_{n\to\infty} n^2 V_n=\frac{1}{3} m^2 \sin ^2(\frac{\pi }{m})$, i.e., the quantization coefficient exists as a finite positive number which equals $\frac{1}{3} m^2 \sin ^2(\frac{\pi }{m})$.
Thus, the proof of the theorem is complete.
\end{proof}

\begin{remark} Since  $\lim\limits_{m\to \infty}2 \sin \frac{\pi }{m}=0$, by \eqref{eq100}, we can conclude that when $m$ tends to $\infty$, then the length of each side of the regular $m$-sided polygon becomes zero, i.e., the regular $m$-sided polygon coincides with the circle. Moreover, for $m\geq 3$, we have
 \[\frac{d}{dm}(\frac{1}{3} m^2 \sin ^2(\frac{\pi }{m}))=\frac{2}{3} \sin (\frac{\pi }{m}) (m \sin (\frac{\pi }{m})-\pi  \cos (\frac{\pi }{m}))>0\]
yielding the fact that the quantization coefficient $\frac{1}{3} m^2 \sin ^2(\frac{\pi }{m})$ for the uniform distribution on the boundary of the regular $m$-sided polygon is an increasing function of $m$. Again,
 \[\lim_{m\to\infty}\frac{1}{3} m^2 \sin ^2(\frac{\pi }{m})=\frac{\pi ^2}{3},\]
 i.e., when $m$ tends to infinity, then the quantization coefficient of the regular $m$-sided polygon equals $\frac{\pi ^2}{3}$.
 Recall that $\frac{\pi ^2}{3}$ is the quantization coefficient for the uniform distribution on the unit circle. Thus, the result in this paper, proves the conjecture given by Pena et al. in the paper \cite{PRRSS}.
 \end{remark}

 \begin{remark}
 If $m=6$, we see that $\lim\limits_{n\to \infty} n^2V_n=3$, which is the quantization coefficient for the uniform distribution on the boundary of a hexagon inscribed in a unit circle. Thus, the result in this paper, also generalizes a result given by Pena et al. in the paper \cite{PRRSS}.
 \end{remark}


\begin{thebibliography}{9999}



    \bibitem[BW]{BW} J.A. Bucklew and G.L. Wise, \emph{Multidimensional asymptotic quantization theory with $r$th power distortion measures}, IEEE Transactions on Information Theory, 1982, Vol. 28, Issue 2, 239-247.


\bibitem[CR1]{CR1} D. \c C\"omez and M.K. Roychowdhury, \emph{Quantization for uniform distributions on stretched Sierpinski triangles}, Monatshefte f\"ur Mathematik, Volume 190, Issue 1,  79-100 (2019).


\bibitem[CR2]{CR2} D. \c C\"omez and M.K. Roychowdhury, \emph{Quantization for uniform distributions of Cantor dusts on $\mathbb{R}^2$}, Topology Proceedings, Volume 56 (2020), Pages 195-218.
%



\bibitem[DR1]{DR1} C.P. Dettmann and M.K. Roychowdhury, \emph{Quantization for uniform distributions on equilateral triangles}, Real Analysis Exchange, Vol. 42(1), 2017, pp. 149-166.

\bibitem[DR2]{DR2} C.P. Dettmann and M.K. Roychowdhury, \emph{An algorithm to compute CVTs for finitely generated Cantor distributions}, to appear, Southeast Asian Bulletin of Mathematics.



\bibitem[GG]{GG} A. Gersho and R.M. Gray, \emph{Vector quantization and signal compression}, Kluwer Academy publishers: Boston, 1992.




\bibitem[GL1]{GL1} S. Graf and H. Luschgy, \emph{Foundations of quantization for probability distributions}, Lecture Notes in Mathematics 1730, Springer, Berlin, 2000.


\bibitem[GL2]{GL2} S. Graf and H. Luschgy, \emph{The Quantization of the Cantor Distribution}, Math. Nachr., 183 (1997), 113-133.


    \bibitem[GN]{GN}  R. Gray and D. Neuhoff, \emph{Quantization}, IEEE Trans. Inform. Theory,  44 (1998), pp. 2325-2383.

%

\bibitem[L]{L} L. Roychowdhury, \emph{Optimal quantization for nonuniform Cantor distributions}, Journal of Interdisciplinary Mathematics, Vol 22 (2019), pp. 1325-1348.
\bibitem[P] {P} K. P\"otzelberger, \emph{The quantization dimension of distributions}, Math. Proc. Camb. Phil. Soc., 131, 507-519 (2001).


\bibitem[PRRSS]{PRRSS} G. Pena, H. Rodrigo, M.K. Roychowdhury, J. Sifuentes, and E. Suazo, \emph{Quantization for uniform distributions on hexagonal, semicircular, and elliptical curves}, Journal of Optimization Theory and Applications,  (2021) 188: 113-142.


\bibitem[R1]{R1} M.K. Roychowdhury, \emph{Quantization and centroidal Voronoi tessellations for probability measures on dyadic Cantor sets}, Journal of Fractal Geometry, 4 (2017), 127-146.


\bibitem[R2]{R2} M.K. Roychowdhury, \emph{Optimal quantizers for some absolutely continuous probability measures}, Real Analysis Exchange, Vol. 43(1), 2017, pp. 105-136.
\bibitem[R3]{R3} M.K. Roychowdhury, \emph{Optimal quantization for the Cantor distribution generated by infinite similitudes},  Israel Journal of Mathematics 231 (2019), 437-466.



\bibitem[R4]{R4} M.K. Roychowdhury, \emph{Least upper bound of the exact formula for optimal quantization of some uniform Cantor distributions}, Discrete and Continuous Dynamical Systems- Series A, Volume 38, Number 9, September 2018, pp. 4555-4570.

\bibitem[R5]{R5} M.K. Roychowdhury, \emph{Center of mass and the optimal quantizers for some continuous and discrete uniform distributions}, Journal of Interdisciplinary Mathematics, Vol. 22 (2019), No. 4,  pp. 451-471.

\bibitem[R6]{R6} M.K. Roychowdhury, \emph{The quantization of the standard triadic Cantor distribution},  Houston Journal of Mathematics, Volume 46, Number 2, 2020, Pages 389-407.

\bibitem[RR1]{RR1} J. Rosenblatt and M.K. Roychowdhury, \emph{Optimal quantization for piecewise uniform distributions}, Uniform Distribution Theory 13 (2018), no. 2, 23-55.

\bibitem[RR2]{RR2} J. Rosenblatt and M.K. Roychowdhury, \emph{Uniform distributions on curves and quantization}, arXiv:1809.08364 [math.PR].


\bibitem[RS]{RS} M.K. Roychowdhury, and Wasiela Salinas, \emph{Quantization for a mixture of uniform distributions associated with probability vectors}, Uniform Distribution Theory 15 (2020), no. 1, 105-142.







\bibitem[Z]{Z} R. Zam, \emph{Lattice Coding for Signals and Networks: A Structured Coding Approach to Quantization, Modulation, and Multiuser Information Theory}, Cambridge University Press, 2014.



\end{thebibliography}
\end{document}